\documentclass[a4paper,11pt,reqno,noindent]{amsart}
\usepackage[centertags]{amsmath}
\usepackage{amsfonts,amssymb,amsthm,dsfont,cases,amscd,esint,enumerate}
\usepackage[T1]{fontenc}
\usepackage[english]{babel}
\usepackage[applemac]{inputenc}
\usepackage{newlfont}
\usepackage{color}
\usepackage[body={15cm,21.5cm},centering]{geometry} 
\usepackage{fancyhdr}
\usepackage{enumerate}

\theoremstyle{plain}
\newtheorem{theor10}{Theorem}

\newtheorem{theor0}{Theorem}[section]
\newenvironment{theor}
  {\pushQED{\qed}\begin{theor0}}
  {\popQED\end{theor0}}
\newtheorem{lem0}[theor0]{Lemma}
\newenvironment{lem}
  {\pushQED{\qed}\begin{lem0}}
  {\popQED\end{lem0}}
\newtheorem{prop0}[theor0]{Proposition}
\newenvironment{prop}
  {\pushQED{\qed}\begin{prop0}}
  {\popQED\end{prop0}}
\newtheorem{cor0}[theor0]{Corollary}

\theoremstyle{definition}
\newtheorem{rem0}[theor0]{Remark}
\newenvironment{rem}
  {\pushQED{\qed}\begin{rem0}}
  {\popQED\end{rem0}}

\mathchardef\emptyset="001F
\numberwithin{equation}{section}

\newcommand{\N}{\mathbb N}
\newcommand{\e}{\varepsilon}
\newcommand{\calG}{\mathcal{G}}
\newcommand{\calQ}{\mathcal{Q}}
\newcommand{\Lc}{\mathcal{L}}
\newcommand{\calC}{\mathcal{C}}
\newcommand{\Sc}{\mathcal S}
\newcommand{\Hf}{\mathfrak H}
\newcommand{\dist}{\operatorname{dist}}
\newcommand{\R}{\mathbb R}
\newcommand{\calL}{\mathcal L}
\newcommand{\Nc}{\mathcal N}
\newcommand{\loc}{{\operatorname{loc}}}
\newcommand{\Id}{\operatorname{Id}}
\newcommand{\E}{\mathbb{E}}

\newcommand{\ee}{e}
\newcommand{\Aa}{\boldsymbol a}
\newcommand{\Ld}{\operatorname{L}}
\newcommand{\step}[1]{\noindent \textit{Step} #1.}

\newcommand{\Pm}{\mathbb{P}}

\newcommand{\expec}[1]{\mathbb{E}\left[ #1 \right]}

\newcommand{\expecmm}[1]{\mathbb{E}[ #1 ]}

\newcommand{\cov}[2]{\operatorname{Cov}\left[{#1};{#2}\right]}

\usepackage[colorlinks,citecolor=black,urlcolor=black]{hyperref}

\title[Eigenvalue fluctuations for random elliptic operators]{Eigenvalue fluctuations for random elliptic operators in homogenization regime}

\author[M. Duerinckx]{Mitia Duerinckx}
\address[Mitia Duerinckx]{Universit\'e Paris-Saclay, CNRS, Laboratoire de Math\'ematiques d'Orsay, 91400~Orsay, France
\& University of California, Los Angeles, Department of Mathematics, CA~90095, USA
\& Universit\'e Libre de Bruxelles, D\'epartement de Math\'ematique,
1050~Brussels, Belgium}
\email{mitia.duerinckx@u-psud.fr}

\begin{document}
\selectlanguage{english}
\begin{abstract}
This work is devoted to the asymptotic behavior of eigenvalues of an elliptic operator with rapidly oscillating random coefficients on a bounded domain with Dirichlet boundary conditions. A sharp convergence rate is obtained for isolated eigenvalues towards eigenvalues of the homogenized problem, as well as a quantitative two-scale expansion result for eigenfunctions. Next, a quantitative central limit theorem is established for eigenvalue fluctuations; more precisely, a pathwise characterization of eigenvalue fluctuations is obtained in terms of the so-called homogenization commutator, in parallel with the recent fluctuation theory for the solution operator.
\end{abstract}

\maketitle

\section{Introduction}
Let $\Aa$ be a stationary and ergodic random coefficient field on $\R^d$ with symmetric values in $\R^{d\times d}$, with the following boundedness and uniform ellipticity properties, for some deterministic constant $\nu>0$,
\begin{equation}\label{eq:ell-as}
|\Aa(x)\ee|\le|\ee|,\qquad \ee\cdot\Aa(x)\ee\ge\nu|\ee|^2,\qquad\text{almost surely, \quad for all $x,\ee\in\R^d$},
\end{equation}
and denote by $(\Omega,\Pm)$ the underlying probability space. In the sequel, we further assume that $\Aa$ satisfies some strong mixing condition, and our main results focus for simplicity on a Gaussian model, see Section~\ref{sec:as}.
Given a bounded $C^{1,1}$-domain $U\subset\R^d$, we consider the sequence of rescaled operators $-\nabla\cdot\Aa(\tfrac\cdot\e)\nabla$ on $H^1_0(U)$. We consider their eigenvalues $\{\lambda_\e^k\}_{k\ge1}$, listed in increasing order and repeated according to multiplicity, and we choose corresponding orthonormal eigenfunctions $\{g_\e^k\}_{k\ge1}\subset H^1_0(U)$,
\begin{equation}\label{eq:eigen-eps}
-\nabla\cdot\Aa(\tfrac\cdot\e)\nabla g_\e^k\,=\,\lambda_\e^kg_\e^k\quad\text{in $U$},\qquad\|g_\e^k\|_{\Ld^2(U)}=1.
\end{equation}
As is well-known, see e.g.~\cite[Section~11]{JKO94}, the eigenvalues $\{\lambda_\e^k\}_{k\ge1}$ converge almost surely to the corresponding eigenvalues $\{\bar\lambda^k\}_{k\ge1}$ of the homogenized operator $-\nabla\cdot\bar\Aa\nabla$,
\begin{equation}\label{eq:eigen-bar}
-\nabla\cdot\bar\Aa\nabla\bar g^k\,=\,\bar\lambda^k\bar g^k\quad\text{in $U$},\qquad\|\bar g^k\|_{\Ld^2(U)}=1,
\end{equation}
where the effective coefficient $\bar\Aa\in\R^{d\times d}$ is defined in each direction $\ee_\alpha$, $1\le\alpha\le d$, by
\[\bar\Aa\ee_\alpha=\expec{\Aa(\nabla\varphi_\alpha+\ee_\alpha)},\]
in terms of the so-called corrector gradient $\nabla\varphi_\alpha$, which is defined as the unique almost sure gradient solution in $\Ld^2_\loc(\R^d)^d$ of the corrector equation
\[-\nabla\cdot\Aa(\nabla\varphi_\alpha+\ee_\alpha)=0,\qquad\text{in $\R^d$},\]
such that $\nabla\varphi_\alpha$ is a stationary field with vanishing expectation and bounded second moment.
In addition, in case of simple eigenvalues, normalized eigenfunctions are also known to converge weakly in $H^1_0(U)$ to the corresponding eigenfunctions of the homogenized operator.
In the present contribution, we further establish sharp convergence rates and analyze random fluctuations.
More precisely, our results are twofold:
\begin{enumerate}[(i)]
\item We prove an optimal convergence rate for simple eigenvalues and give a two-scale description of eigenfunctions. Such results were only known previously in the simpler periodic setting~\cite{Kesavan-1,Kesavan-2}.
\smallskip\item We characterize joint fluctuations of simple eigenvalues in form of a quantitative central limit theorem, thus answering a question raised in~\cite[Section~2.2]{BFK-16}.
More precisely, we further unravel the pathwise structure of fluctuations: in the spirit of our recent work with Gloria and Otto~\cite{DGO1} (see also the related heuristics by Armstrong, Gu, and Mourrat in~\cite{GuM}), we show that fluctuations of~$\lambda_\e^k$ are pointwise close to fluctuations of~$\int_U\Xi_{\alpha\beta}^\circ(\tfrac\cdot\e)\partial_\alpha\bar g^k\partial_\beta\bar g^k$,\footnote{Throughout, we use Einstein's convention of summation on repeated indices, here on $\alpha,\beta$.} in terms of the so-called {standard homogenization commutator}
\begin{equation}\label{eq:commut}
\quad\Xi_{\alpha\beta}^\circ:=\ee_\beta\cdot(\Aa-\bar\Aa)(\nabla\varphi_\alpha+\ee_\alpha).
\end{equation}
In other words, while we found in~\cite{DGO1,DGO2} that $\Xi^\circ$ governs fluctuations of the solution operator, we show in the present contribution that this quantity further governs eigenvalue fluctuations.
The characterization of the latter is a direct consequence of this pathwise relation combined with our study of the scaling limit of~$\Xi^\circ$ in~\cite{DGO1,DO1,DFG1}.
\end{enumerate}
These different results make a heavy use of refined tools from the recent quantitative theory of stochastic homogenization as developed in~\cite{AKM-book,GNO-reg,GNO-quant,DO1}.

\medskip
We briefly explain how our fluctuation result relates to the spectral statistics conjecture for random operators.
Rescaling the eigenvalue relation~\eqref{eq:eigen-eps}, eigenvalues of the operator~$-\nabla\cdot\Aa\nabla$ on the dilated domain $\tfrac1\e U$ coincide with $\{\e^2\lambda_\e^k\}_{k\ge1}$, and we consider the large-volume limit $\e\downarrow0$.
In this contribution, we show that the first eigenvalues have joint Gaussian fluctuations, in the sense that the vector \mbox{$\e^{-d/2}(\lambda_\e^1-\expecmm{\lambda_\e^1},\ldots,\lambda_\e^n-\expecmm{\lambda_\e^n})$} is asymptotically Gaussian for any fixed~$n$.
This Gaussian fluctuation result at the bottom of the spectrum is new and should be compared to the conjecture that eigenvalues have local Poisson statistics in spectral regions where localization holds (in particular, at edges of the spectrum other than its bottom) and have random matrix GOE statistics in the bulk of regions where delocalization holds.
Rigorous results on Poisson statistics in the localized regime were pioneered by Minami~\cite{Minami} for the Anderson model, and we refer to~\cite{GK14,Hislop-Krishna,Dietlein-Elgart} and references therein for recent developments, but to our knowledge the problem still remains open for the divergence-form operator $-\nabla\cdot\Aa\nabla$ apart from the 1D case covered in~\cite{Shirley-thesis}. Rigorous results on GOE statistics in the delocalized regime are only known in the simplified setting of random band matrix models~\cite{BYY,Bourgade}.

\medskip
The article is organized as follows. Precise assumptions and main results are stated in Section~\ref{sec:res}.
We focus for simplicity on a Gaussian model for the coefficient fields $\Aa$, in which case Malliavin calculus is available and simplifies the analysis.
In Section~\ref{sec:tools}, we recall some useful tools from the quantitative theory of stochastic homogenization, including corrector estimates and large-scale regularity theory, and we recall notations from Malliavin calculus. Proofs of the main results are postponed to Section~\ref{sec:proofs}.

\medskip
\subsection*{Notation}
\begin{enumerate}[$\bullet$]
\item We denote by $C\ge1$ any constant that only depends on the dimension $d$, on the ellipticity constant $\nu$ in~\eqref{eq:ell-as}, on the domain $U$, and on $\|a_0\|_{W^{2,\infty}}$ and $\int_{\R^d}[\calC_0]_\infty$ in~\eqref{eq:def-A} and~\eqref{eq:integr} below. We use the notation $\lesssim$ (resp.~$\gtrsim$) for $\le C\times$ (resp.~$\ge\frac1C\times$) up to such a multiplicative constant~$C$. We write $\simeq$ when both $\lesssim$ and $\gtrsim$ hold. We add subscripts to $C,\lesssim,\gtrsim,\simeq$ to indicate dependence on other parameters.
\smallskip\item We denote by $B_r(x)$ the ball of radius $r$ centered at $x$ in $\R^d$, and we write for shortness $B_r:=B_r(0)$, $B(x):=B_1(x)$, and $B:=B_1(0)$.
\smallskip\item For a function $g$ and an exponent $1\le p<\infty$, we write $[g]_p(x):=(\fint_{B(x)}|g|^p)^{1/p}$ for the local moving $\Ld^p$-averages, and similarly $[g]_\infty(x):=\sup_{B(x)}|g|$. For averages at the scale~$\e$, we write $[g]_{p;\e}(x):=(\fint_{B_\e(x)}|g|^p)^{1/p}$.
\end{enumerate}

\section{Main results}\label{sec:res}

\subsection{Assumptions}\label{sec:as}
Let $U\subset\R^d$ be a bounded $C^{1,1}$-domain.
For the random coefficient field $\Aa$, we focus on a Gaussian model: more precisely, we set
\begin{equation}\label{eq:def-A}
\Aa(x)\,:=\,a_0(G(x)),
\end{equation}
where $a_0\in C^2_b(\R^\kappa)^{d\times d}$ is such that the boundedness and uniform ellipticity requirements~\eqref{eq:ell-as} are pointwise satisfied, and where $G:\R^d\times\Omega\to\R^\kappa$ is an $\R^\kappa$-valued centered stationary Gaussian random field on $\R^d$ with covariance function $\calC:\R^d\to\R^{\kappa\times\kappa}$, constructed on a probability space $(\Omega,\Pm)$. In addition, we assume that $G$ has integrable correlations in the following sense: starting from the representation
\[G_i=\calC_{0;ij}\ast\xi_j,\]
where $\xi$ is an $\R^\kappa$-valued Gaussian white noise on $\R^d$ and where the kernel $\calC_0:\R^d\to\R^{\kappa\times \kappa}$ satisfies $\calC_{0;il}\ast\calC_{0;lj}=\calC_{ij}$, we assume that $\calC_0$ satisfies the integrability condition
\begin{equation}\label{eq:integr}
\int_{\R^d}\,[\calC_0]_\infty\,<\,\infty.
\end{equation}
In particular, this entails that the covariance function $\calC$ itself satisfies the same integrability condition $\int_{\R^d}[\calC]_\infty\!<\!\infty$. Moreover, $\calC$ is necessarily continuous, so that $G$ and $\Aa$ are stochastically continuous and jointly measurable on $\R^d\times\Omega$.

\begin{rem}[Relaxation of assumptions]
This Gaussian model~\eqref{eq:def-A}--\eqref{eq:integr} allows to exploit Malliavin calculus techniques, which strongly simplifies the analysis. Our approach can be repeated mutatis mutandis in a corresponding Poisson model or in the iid discrete setting, using corresponding stochastic calculus techniques, e.g.~\cite{Peccati-Reitzner,D-20a}.
It can be further adapted to the case of a degraded stochastic calculus in form of multiscale variance inequalities as we introduced in~\cite{DG20a,DG20b} with Gloria, which are available for a much wider class of mixing coefficient fields.
The general case of an $\alpha$-mixing coefficient field is however much more demanding: we believe that it can be treated using the recent techniques of~\cite{AKM-book,GO4}, but we do not pursue in that direction here.
Finally, the integrability condition~\eqref{eq:integr} is easily relaxed:  the Gaussian model with non-integrable correlations can be treated similarly but would yield different scalings as in~\cite{GNO-quant,DGO2,DFG1}.
\end{rem}

\subsection{Convergence rate for eigenvalues and eigenfunctions}
The following result provides a sharp convergence rate for simple eigenvalues, as well as a quantitative two-scale expansion for corresponding eigenfunctions.
The square root in the convergence rate for eigenfunctions in~\eqref{eq:lam-conv1}--\eqref{eq:lam-conv2} is due to boundary layers.

\begin{theor}\label{th:conv}
For all $k\ge1$ such that $\bar\lambda^k$ is simple, we have for all $q<\infty$,
\begin{eqnarray}
\|\lambda_\e^k-\bar\lambda^k\|_{\Ld^q(\Omega)}&\lesssim_{k,q}&\e\mu_d(\tfrac1\e),\label{eq:lam-conv}\\
\|g_\e^k-\bar g^k\|_{\Ld^q(\Omega;\Ld^2(U))}&\lesssim_{k,q}&(\e\mu_d(\tfrac1\e))^\frac12,\label{eq:lam-conv1}\\
\|\nabla g_\e^k-(\nabla\varphi_\alpha+\ee_\alpha)(\tfrac\cdot\e)\partial_\alpha\bar g^k\|_{\Ld^q(\Omega;\Ld^2(U))}&\lesssim_{k,q}&(\e\mu_d(\tfrac1\e))^\frac12,\label{eq:lam-conv2}
\end{eqnarray}
in terms of
\begin{equation}\label{eq:mudr}
\mu_d(r)\,:=\,\left\{\begin{array}{lll}
1&:&d>2,\\
\log(2+r)^\frac12&:&d=2,\\
(1+r)^\frac12&:&d=1.
\end{array}\right.\qedhere
\end{equation}
\end{theor}

\subsection{Eigenvalue fluctuations}
The following result shows that eigenvalue fluctuations are governed to leading order by fluctuations of the so-called standard homogenization commutator~\eqref{eq:commut}. Combined with the scaling limit for the latter in~\cite{DGO1,DO1,DFG1}, this yields a full characterization of eigenvalue fluctuations together with a convergence rate.

\begin{theor}\label{th:fluct}
For all $k\ge1$ such that $\bar\lambda^k$ is simple, we have for all $q<\infty$,
\begin{equation}\label{eq:pathwise}
\e^{-\frac d2}\Big\|\lambda_\e^k-\expecmm{\lambda_\e^k}-\int_U\Xi_{\alpha\beta}^\circ(\tfrac\cdot\e)\partial_\alpha\bar g^k\partial_\beta\bar g^k\Big\|_{\Ld^q(\Omega)}\,\lesssim_{k,q}\,(\e\mu_d(\tfrac1\e))^\frac12,
\end{equation}
where we recall that the standard homogenization commutator $\Xi^\circ$ is defined in~\eqref{eq:commut}, and where $\mu_d$ is given by~\eqref{eq:mudr}.
Combined with the known scaling limit for $\Xi^\circ$, cf.~\cite{DGO1,DO1,DFG1}, this yields for all $k_1,\ldots, k_n\ge1$ such that $\bar\lambda^{k_1},\ldots,\bar\lambda^{k_n}$ are simple,
\[W_2\bigg({\e^{-\frac d2}\Big(\big(\lambda_\e^{k_1}-\expecmm{\lambda_\e^{k_1}}\big),\ldots,\big(\lambda_\e^{k_n}-\expecmm{\lambda_\e^{k_n}}\big)\Big)\,};{\,\Nc_{k_1,\ldots,k_n}}\bigg)\,\lesssim_{k_1,\ldots,k_n}\,(\e\mu_d(\tfrac1\e))^\frac12,\]
where $W_2(\cdot;\cdot)$ denotes the $2$-Wasserstein distance and where $\Nc_{k_1,\ldots,k_n}$ stands for the $n$-dimensional centered Gaussian vector with covariance
\[\expec{(\Nc_{k_1,\ldots,k_n})_i(\Nc_{k_1,\ldots,k_n})_j}\,=\,\int_{\R^d}(\nabla\bar g^{k_i}\otimes \nabla\bar g^{k_i}):\calQ\,(\nabla\bar g^{k_j}\otimes\nabla\bar g^{k_j}),\]
where the $4$th-order tensor $\calQ\in\R^{d\times d\times d\times d}$ is given by the following Green--Kubo formula, for any cut-off function $\chi\in C^\infty_c(\R^d)$ with $\chi(0)=1$,
\begin{equation}\label{eq:GK}
\calQ_{\alpha'\beta'\alpha\beta}\,:=\,\lim_{L\uparrow\infty}\int_{\R^d}\chi(\tfrac1Lx)\,\cov{\Xi_{\alpha'\beta'}^\circ(0)}{\Xi_{\alpha\beta}^\circ(x)}dx.\qedhere
\end{equation}
\end{theor}

\begin{rem}
As shown in~\cite{DGO1,DO1,DFG1}, although the covariance function of the homogenization commutator $\Xi^\circ$ is only borderline integrable,
\[\big|\!\cov{\Xi_{\alpha'\beta'}^\circ(0)}{\Xi_{\alpha\beta}^\circ(x)}\!\big|\,\lesssim\,(1+|x|)^{-d},\]
the limit~\eqref{eq:GK} indeed exists and the convergence holds with rate $O(L^{-1}\mu_d(L))$.
Alternatively, in terms of Malliavin calculus, the effective tensor~$\calQ$ can be expressed as
\begin{multline*}
\calQ_{\alpha'\beta'\alpha\beta}\,:=\int_{\R^d}\calC_{ij}(y)\,\,\E\Big[\big((\nabla\varphi_{\beta'}+\ee_{\beta'})\cdot \partial_ia_0(G)(\nabla\varphi_{\alpha'}+\ee_{\alpha'})\big)(0)\label{eq:def-Q}\\
\times(\Lc+1)^{-1}\big((\nabla\varphi_{\beta}+\ee_\beta)\cdot \partial_ja_0(G)(\nabla\varphi_{\alpha}+\ee_\alpha)\big)(y)\Big]\,dy,
\end{multline*}
where $\Lc$ is the Ornstein--Uhlenbeck operator  associated with the Malliavin calculus with respect to the underlying Gaussian field $G$, cf.~\eqref{eq:def-OU} below.
\end{rem}

\section{Main tools}\label{sec:tools}
In this section, we recall useful tools both from the quantitative theory of stochastic homogenization, including corrector estimates and large-scale regularity theory, and from Malliavin calculus.

\subsection{Tools from quantitative homogenization theory}
The following result recalls the definition of correctors and flux corrector, e.g.~\cite[Lemma~1]{GNO-reg}, which are key to describe fine oscillations of the solution operator. Note that the flux corrector $\sigma_\alpha$ is defined as a vector potential for the flux $q_\alpha=\Aa(\nabla\varphi_\alpha+\ee_\alpha)-\bar\Aa\ee_\alpha$, cf.~\eqref{eq:prop-sig}, and the defining equation~\eqref{eq:def-sig} amounts to choosing the Coulomb gauge.

\begin{lem}[Correctors; \cite{GNO-reg}]\label{lem:cor}
For all $1\le \alpha\le d$, there exists a unique solution $\varphi_\alpha$ to the following infinite-volume corrector problem:
\begin{enumerate}[\quad$\bullet$]
\item Almost surely, $\varphi_\alpha$ belongs to $H^1_\loc(\R^d)$ and satisfies in the weak sense
\[\quad-\nabla\cdot\Aa(\nabla\varphi_\alpha+\ee_\alpha)\,=\,0,\qquad\text{in $\R^d$}.\]
\item The gradient field $\nabla\varphi_\alpha$ is stationary, has vanishing expectation, and has bounded second moment, and $\varphi_\alpha$ satisfies the anchoring condition $\fint_B\varphi_\alpha=0$ almost surely.
\end{enumerate}
In addition, there exists a unique random $2$-tensor field $\sigma_\alpha=\{\sigma_{\alpha\beta\gamma}\}_{1\le \beta,\gamma\le d}$ that satisfies the following infinite-volume problem:
\begin{enumerate}[\quad$\bullet$]
\item For all $1\le\beta,\gamma\le d$, almost surely, $\sigma_{\alpha\beta\gamma}$ belongs to $H^1_\loc(\R^d)$ and satisfies in the weak sense
\begin{equation}\label{eq:def-sig}
\quad-\triangle\sigma_{\alpha\beta\gamma}\,=\,\partial_\beta (q_{\alpha})_\gamma-\partial_\gamma (q_{\alpha})_\beta,\qquad\text{in $\R^d$},
\end{equation}
in terms of the flux $q_\alpha:=\Aa(\nabla\varphi_\alpha+\ee_\alpha)-\bar\Aa\ee_\alpha$.
\smallskip\item The gradient field $\nabla\sigma_\alpha$ is stationary, has vanishing expectation, and has bounded second moment, and $\sigma_\alpha$ satisfies the anchoring condition $\fint_B\sigma_\alpha=0$ almost surely.
\end{enumerate}
In particular, this definition entails
\begin{equation}\label{eq:prop-sig}
\nabla\cdot\sigma_\alpha=q_\alpha,\qquad\sigma_{\alpha\beta\gamma}=-\sigma_{\alpha\gamma\beta}.
\qedhere
\end{equation}
\end{lem}

Next, in the present Gaussian setting~\eqref{eq:def-A}--\eqref{eq:integr}, we have the following moment bounds on corrector gradients, as well as optimal estimates on the sublinearity of correctors, see~\cite{AKM-book,GNO-quant}.
In dimension $d>2$, these estimates ensure that correctors $\varphi,\sigma$ can be chosen themselves as stationary fields.

\begin{theor}[Corrector estimates; \cite{AKM-book,GNO-quant}]\label{th:cor}
For all $q<\infty$,
\[\|[\nabla\varphi]_2\|_{\Ld^q(\Omega)}+\|[\nabla\sigma]_2\|_{\Ld^q(\Omega)}\,\lesssim_q\,1,\]
and for all $x\in\R^d$,
\[\|[\varphi]_2(x)\|_{\Ld^q(\Omega)}+\|[\sigma]_2(x)\|_{\Ld^q(\Omega)}\,\lesssim_q\,\mu_d(|x|),\]
where we recall that $\mu_d$ is given by~\eqref{eq:mudr}.
In addition, the following Meyers-type improvement holds: there exists a constant $C_0\simeq1$ such that for all~$2\le p\le2+\frac1{C_0}$ the local quadratic averages $[\cdot]_2$ in the above estimates can be replaced by $[\cdot]_p$.
\end{theor}

A key insight in quantitative stochastic homogenization theory is the idea of large-scale regularity, which started with Avellaneda and Lin~\cite{Avellaneda-Lin-87} in the periodic setting, then with Armstrong and Smart~\cite{AS} in the random setting, and was fully developed in recent years in~\cite{Armstrong-Mourrat-16,AKM-book,GNO-reg}: due to homogenization, the heterogeneous elliptic operator \mbox{$-\nabla\cdot\Aa\nabla$} can be expected to inherit the same regularity properties as its homogenized version $-\nabla\cdot\bar\Aa\nabla$ on large scales.
For our purpose in this work, we focus on large-scale $\Ld^p$-regularity and we appeal to a convenient annealed version that we established in~\cite[Theorem~6.1]{DO1} with Otto.
More precisely, while in~\cite{DO1} only interior $\Ld^p$-regularity was established, the following is further stated to hold globally on any bounded domain with Dirichlet boundary conditions: the proof follows as in~\cite[Section~6]{DO1} up to replacing the use of large-scale interior Lipschitz regularity by corresponding global regularity as developed in~\cite[Section~3.5]{AKM-book} (see also~\cite{Fischer-Raithel}).

\begin{theor}[Annealed $\Ld^p$-regularity; \cite{DO1,AKM-book}]\label{th:ann-reg}
Let $D\subset\R^d$ be a bounded $C^{1,\gamma}$-domain for some $\gamma>0$.
For all $0<\e\le1$ and $h\in C^\infty_c(D;\Ld^\infty(\Omega)^d)$, if $u_{\e;h}$ is almost surely the unique solution in $H^1_0(D)$ of
\[-\nabla\cdot\Aa(\tfrac\cdot\e)\nabla u_{\e;h}=\nabla\cdot h,\qquad\text{in $U$},\]
then there holds for all $1<p,q<\infty$ and $\delta>0$,
\[\|[\nabla u_{\e;h}]_{2;\e}\|_{\Ld^p(D;\Ld^q(\Omega))}\,\lesssim_{D,p,q,\delta}\,\|[h]_{2;\e}\|_{\Ld^p(D;\Ld^{q+\delta}(\Omega))}.\qedhere\]
\end{theor}

\subsection{Tools from Malliavin calculus}
We recall some classical notation and tools from Malliavin calculus; we refer e.g.\@ to~\cite{NP-08} for details.
We set
\[\calG(\zeta):=\int_{\R^d}G\cdot\zeta,\qquad\text{for all $\zeta\in C^\infty_c(\R^d)^\kappa$,}\]
which are jointly Gaussian random variables with covariance
\[\cov{\calG(\zeta)}{\calG(\zeta')}\,=\,\iint_{\R^d\times\R^d}\calC_{ij}(x-y)\,\zeta_i(x)\zeta_j'(y)\,dxdy, \qquad\zeta,\zeta'\in C^\infty_c(\R^d)^\kappa.\]
Defining $\Hf$ as the closure of $C^\infty_c(\R^d)^\kappa$ for this (semi)norm,
\[\|\zeta\|_\Hf^2:=\langle\zeta,\zeta\rangle_{\Hf},\qquad\langle\zeta,\zeta'\rangle_{\Hf}:=\iint_{\R^d\times\R^d}\calC_{ij}(x-y)\,\zeta_i(x)\zeta'_j(y)\,dxdy,\]
we may extend by density the definition of $\calG(\zeta)\in\Ld^2(\Omega)$ to all $\zeta\in\Hf$.
The space~$\Hf$ (up to taking the quotient with respect to the kernel of $\|\cdot\|_\Hf$) is a separable Hilbert space and embeds isometrically into $\Ld^2(\Omega)$ via $\zeta\mapsto\calG(\zeta)$.
In view of the integrability condition~\eqref{eq:integr}, the norm of $\Hf$ is bounded by
\begin{equation}\label{eq:integr-cons}
\|\zeta\|_\Hf\,\lesssim\,\|[\zeta]_1\|_{\Ld^2(\R^d)}.
\end{equation}
Without loss of generality we can assume that the probability space is endowed with the $\sigma$-algebra generated by the Gaussian field $G$, so that the linear subspace
\[\Sc(\Omega)\,:=\,\Big\{g(\calG(\zeta_1),\ldots,\calG(\zeta_n))\,:\,n\in\N,\,g\in C^\infty_c(\R^n),\,\zeta_1,\ldots,\zeta_n \in \Hf\Big\}\]
is dense in $\Ld^2(\Omega)$. We may thus define operators on this simpler subspace $\Sc(\Omega)$ before extending them by density to $\Ld^2(\Omega)$.
For a random variable $X\in\Sc(\Omega)$, say $X=g(\calG(\zeta_1),\ldots,\calG(\zeta_n))$, we define its {Malliavin derivative} $DX\in\Ld^2(\Omega;\Hf)$ as
\[DX\,:=\,\sum_{i=1}^n\zeta_i\,(\partial_ig)(\calG(\zeta_1),\ldots,\calG(\zeta_n)).\]
We can check that this operator $D:\Sc(\Omega)\subset\Ld^2(\Omega)\to\Ld^2(\Omega;\Hf)$ is closable, and we still denote by $D$ its closure.
Next, we define the divergence operator $D^*$ as the adjoint of $D$, and we construct the so-called {Ornstein--Uhlenbeck operator}
\begin{equation}\label{eq:def-OU}
\calL\,:=\,D^*D,
\end{equation}
which is well-defined as an essentially self-adjoint nonnegative operator on $\Sc(\Omega)\subset\Ld^2(\Omega)$.
With this notation, we may now state the following useful classical result; a short proof and relevant references can be found e.g.\@ in~\cite[Proposition~4.1]{DO1}.

\begin{prop}[Malliavin--Poincaré inequality]\label{prop:Mall}
For all $X\in \Sc(\Omega)$
and $q<\infty$,
\[\|X-\expec{X}\!\|_{\Ld^{2q}(\Omega)}\,\lesssim\,q^\frac12\|DX\|_{\Ld^{2q}(\Omega;\Hf)}.\qedhere\]
\end{prop}

\section{Proof of main results}\label{sec:proofs}

This section is devoted to the proof of our main results. After a few preliminary estimates, Theorems~\ref{th:conv} and~\ref{th:fluct} are established in Sections~\ref{sec:conv} and~\ref{sec:fluct}, respectively.

\subsection{Preliminary estimates}

The following lemma provides uniform bounds on eigenvalues and eigenfunctions.
Uniform bounds on gradients of eigenfunctions, cf.~\eqref{eq:bnd-nabg}, are based on large-scale regularity theory.

\begin{lem}\label{lem:lam-g-nabg}
For all $k\ge1$, we have almost surely,
\begin{equation}\label{eq:bnd-lambda}
\lambda_\e^k\,\simeq\, |k|^2,
\end{equation}
\begin{equation}\label{eq:bnd-g}
|g_\e^k|\,\lesssim_k\,1,
\end{equation}
and for all $1<p,q<\infty$,
\begin{equation}\label{eq:bnd-nabg}
\|[\nabla g_\e^k]_{2;\e}\|_{\Ld^p(U;\Ld^q(\Omega))}\,\lesssim_{k,p,q}\,1.
\end{equation}
In addition, the following Meyers-type improvement holds: there exists a constant $C_0\simeq1$ (independent of $\e,k$) such that for all~$2\le p\le2+\frac1{C_0}$ the local quadratic averages $[\cdot]_{2;\e}$ in~\eqref{eq:bnd-nabg} can be replaced by $[\cdot]_{p;\e}$.
\end{lem}

\begin{proof}
We split the proof into three steps.

\medskip
\step1 Proof of deterministic estimates~\eqref{eq:bnd-lambda} and~\eqref{eq:bnd-g}.\\
The first estimate~\eqref{eq:bnd-lambda} follows from a spectral comparison argument based on the uniform ellipticity condition~\eqref{eq:ell-as}.
We turn to the proof of~\eqref{eq:bnd-g} and we appeal to a similar reproducing kernel trick as in~\cite{BFK-16}: the eigenvalue relation~\eqref{eq:eigen-eps} yields for all $k\ge1$ and~$t\ge0$,
\begin{equation*}
g_\e^k\,=\,e^{\lambda_\e^k t}P_\e^tg_\e^k,
\end{equation*}
in terms of the Dirichlet semigroup $P_\e^t:=e^{t\nabla\cdot\Aa(\frac\cdot\e)\nabla}$ on $U$. Noting that the latter is bounded by the corresponding whole-space semigroup, and appealing to the Nash--Aronson estimates, we deduce almost surely
\[|g_\e^k|\,\lesssim\,e^{\lambda_\e^k t}\int_{U}t^{-\frac d2}\exp(-\tfrac{\nu}{8t}|\cdot-y|^2)\,|g_\e^k(y)|\,dy.\]
Choosing $t=1$, using~\eqref{eq:bnd-lambda}, and recalling that $g_\e^k$ is normalized, we conclude
\[|g_\e^k|\,\lesssim_k\,1.\]

\medskip
\step2 Proof of~\eqref{eq:bnd-nabg}.\\
Considering the unique solution $h_\e^k\in H^1_0(U)$ of the Laplace equation
\begin{equation}\label{eq:def-heps}
\triangle h_\e^k=\lambda_\e^k g_\e^k\qquad\text{in $U$},
\end{equation}
we can rewrite the eigenvalue relation~\eqref{eq:eigen-eps} as
\[-\nabla\cdot\Aa(\tfrac\cdot\e)\nabla g_\e^k=\nabla\cdot(\nabla h_\e^k)\qquad\text{in $U$}.\]
Appealing to annealed $\Ld^p$-regularity in form of Theorem~\ref{th:ann-reg}, we deduce for all $1<p,q<\infty$ and $\delta>0$,
\[\|[\nabla g_\e^k]_{2;\e}\|_{\Ld^p(U;\Ld^q(\Omega))}\,\lesssim_{p,q,\delta}\|[\nabla h_\e^k]_{2;\e}\|_{\Ld^p(U;\Ld^{q+\delta}(\Omega))}\,\le\,\|\nabla h_\e^k\|_{\Ld^\infty(\Omega;\Ld^\infty(U))}.\]
Since Schauder regularity theory applied to equation~\eqref{eq:def-heps} yields almost surely
\[\|\nabla h_\e^k\|_{\Ld^\infty(U)}\,\lesssim\,\lambda_\e^k\|g_\e^k\|_{\Ld^2(U)}\,\lesssim_k\,1,\]
the conclusion~\eqref{eq:bnd-nabg} follows.

\medskip
\step3 Proof of the Meyers-type improvement.\\
For all $0<r\le1$ and $x\in U$, choose a cut-off function $\chi_{r,x}\in C^\infty_c(\R^d)$ with
\[\chi_{r,x}|_{B_r(x)}=1,\qquad\chi_{r,x}|_{\R^d\setminus B_{2r}(x)}=0,\qquad0\le\chi_{r,x}\le1,\qquad\text{and}\qquad|\nabla\chi_{r,x}|\lesssim\tfrac1r.\]
Testing the eigenvalue relation~\eqref{eq:eigen-eps} with $\chi_{r,x}^2(g_\e^k-\fint_{B_{2r}(x)\cap U}g_\e^k)$, and using the properties of the cut-off function $\chi_{r,x}$, we are easily led to the following Caccioppoli estimate,
\[\int_{B_r(x)\cap U}|\nabla g_\e^k|^2\,\lesssim\,\frac1{r^2}\int_{B_{2r}(x)\cap U}\Big|g_\e^k-\fint_{B_{2r}(x)\cap U}g_\e^k\Big|^2+(\lambda_\e^k)^2\int_{B_{2r}(x)\cap U}|g_\e^k|^2.\]
Appealing to the Poincaré--Sobolev inequality, and noting that the regularity of the domain~$U$ ensures $|B_r(x)\cap U|\simeq r^d$ for $r\le1$ and $x\in U$, we deduce
\[\Big(\fint_{B_r(x)\cap U}|\nabla g_\e^k|^2\Big)^\frac12\,\lesssim\,\Big(\fint_{B_{2r}(x)\cap U}|\nabla g_\e^k|^{\frac{2d}{d+2}}\Big)^\frac{d+2}{2d}+\lambda_\e^k\Big(\fint_{B_{2r}(x)\cap U}|g_\e^k|^2\Big)^\frac12.\]
An application of Gehring's lemma~\cite[Proposition~5.1]{Giaquinta-Modica} then ensures the existence of some $C_0\simeq1$ (independent of $\e,k$) such that for all $2\le p\le2+\frac1{C_0}$, for all $0<r\le1$ and $x\in U$,
\[\Big(\fint_{B_r(x)\cap U}|\nabla g_\e^k|^p\Big)^\frac1p\,\lesssim\,\Big(\fint_{B_{2r}(x)\cap U}|\nabla g_\e^k|^2\Big)^\frac12+\lambda_\e^k\Big(\fint_{B_{2r}(x)\cap U}|g_\e^k|^p\Big)^\frac1p.\]
Combined with~\eqref{eq:bnd-lambda}, \eqref{eq:bnd-g}, and~\eqref{eq:bnd-nabg}, this yields the conclusion.
\end{proof}

The following lemma concerns the regularity of eigenfunctions of the homogenized operator.
The proof follows from global $\Ld^p$-regularity theory in a $C^{1,1}$-domain, e.g.~\cite[Theorem~9.13]{Gilbarg-Trudinger-01}, applied to the eigenvalue relation~\eqref{eq:eigen-bar}.

\begin{lem}\label{lem:barg}
For all $k\ge1$, we have for all $1<p<\infty$,
\[\|\bar g^k\|_{W^{2,p}(U)}\,\lesssim_{k,p}\,1.\qedhere\]
\end{lem}

Next, we establish energy-type estimates on the orthogonal complement of a given eigenspace. The multiplicative constant in the estimate depends on the distance between neighboring eigenvalues, cf.~$\delta_\e^k$ in~\eqref{eq:concl-zL2}.

\begin{lem}\label{lem:prelim}
Given $k\ge1$, let $\pi_\e^k$ denote the orthogonal projection $\pi_\e^k[f]:=(\int_Ufg_\e^k)g_\e^k$ on the $k$th eigenspace.
For all $h\in C^\infty_b(U)^d$, if $u_{\e;h}^k\in H^1_0(U)$ satisfies in $U$,
\begin{equation}\label{eq:z}
(-\lambda_\e^k-\nabla\cdot\Aa(\tfrac\cdot\e)\nabla)u_{\e;h}^k\,=\,(\Id-\pi_\e^k)[\nabla\cdot h],\qquad \pi_\e^k[u_{\e;h}^k]=0,
\end{equation}
then there holds
\begin{equation}\label{eq:concl-zL2}
\|\nabla u_{\e;h}^k\|_{\Ld^2(U)}\,\lesssim_k\,(\delta_\e^k)^{-1}\|h\|_{\Ld^2(U)},
\end{equation}
in terms of $\delta_\e^k:=(\lambda_\e^{k+1}-\lambda_\e^k)\wedge(\lambda_\e^k-\lambda_\e^{k-1})\wedge1$.
In addition, the following Meyers-type improvement holds: there exists a constant $C_0\simeq1$ (independent of $\e,k$) such that for all~$2\le p\le2+\frac1{C_0}$ the $\Ld^2(U)$-norms in~\eqref{eq:concl-zL2} can be replaced by $\Ld^p(U)$-norms.
\end{lem}

\begin{proof}
We split the proof into two steps.

\medskip
\step1 Proof of~\eqref{eq:concl-zL2}.\\
Testing equation~\eqref{eq:z} with $u_{\e;h}^k$ itself,
and using~\eqref{eq:bnd-lambda},
we find
\begin{equation}\label{eq:bnd-nabz/z}
\|\nabla u_{\e;h}^k\|_{\Ld^2(U)}\,\lesssim_k\,\|u_{\e;h}^k\|_{\Ld^2(U)}+\|h\|_{\Ld^2(U)},
\end{equation}
and it remains to estimate the $\Ld^2(U)$-norm of $u_{\e;h}^k$.
Inverting equation~\eqref{eq:z} on eigenspaces yields
\begin{equation}\label{eq:form-z}
u_{\e;h}^k\,=\,-\sum_{j:j\ne k}\frac1{\lambda_\e^j-\lambda_\e^k}\Big(\int_Uh\cdot\nabla g_\e^j\Big)g_\e^j.
\end{equation}
Taking the $\Ld^2$-norm and recalling that $\{g_\e^j\}_j$ is an orthonormal system in $\Ld^2(U)$, we deduce
\begin{equation*}
\|u_{\e;h}^k\|_{\Ld^2(U)}^2\,=\,\sum_{j:j\ne k}\frac1{(\lambda_\e^j-\lambda_\e^k)^2}\Big(\int_Uh\cdot\nabla g_\e^j\Big)^2.
\end{equation*}
In terms of the solution $v_{\e;h}\in H^1_0(U)$ of the auxiliary problem
\begin{equation}\label{eq:vepsh}
-\nabla\cdot\Aa(\tfrac\cdot\e)\nabla v_{\e;h}\,=\,\nabla\cdot h\qquad\text{in $U$},
\end{equation}
we can write
\[\int_Uh\cdot\nabla g_\e^j\,=\,-\int_U\nabla v_{\e;h}\cdot\Aa(\tfrac\cdot\e)\nabla g_\e^j\,=\,-\lambda_\e^j\int_Uv_{\e;h}\,g_\e^j.\]
so that the above becomes
\begin{equation*}
\|u_{\e;h}^k\|_{\Ld^2(U)}^2\,=\,\sum_{j:j\ne k}\frac{(\lambda_\e^j)^2}{(\lambda_\e^j-\lambda_\e^k)^2}\Big(\int_Uv_{\e;h}\,g_\e^j\Big)^2.
\end{equation*}
In terms of $\delta_\e^k=(\lambda_\e^{k+1}-\lambda_\e^k)\wedge(\lambda_\e^k-\lambda_\e^{k-1})\wedge1$, using that $\{g_\e^j\}_j$ constitutes an orthonormal basis of $\Ld^2(U)$, we deduce
\begin{equation*}
\|u_{\e;h}^k\|_{\Ld^2(U)}^2\,\lesssim\,(\delta_\e^k)^{-2}\sum_{j}\Big(\int_Uv_{\e;h}g_\e^j\Big)^2\,=\,(\delta_\e^k)^{-2}\|v_{\e;h}\|_{\Ld^2(U)}^2.
\end{equation*}
and thus, by Poincaré's inequality combined with an energy estimate for~\eqref{eq:vepsh},
\begin{equation*}
\|u_{\e;h}^k\|_{\Ld^2(U)}\,\lesssim\,(\delta_\e^k)^{-1}\|h\|_{\Ld^2(U)}.
\end{equation*}
Inserting this into~\eqref{eq:bnd-nabz/z}, the conclusion~\eqref{eq:concl-zL2} follows.

\medskip
\step2 Proof of the Meyers-type improvement.\\
Rewriting equation~\eqref{eq:z} for~$u_{\e;h}^k$ as
\[-\triangle u_{\e;h}^k=\nabla\cdot\big((\tfrac2{1+\nu}\Aa(\tfrac\cdot\e)-\Id)\nabla u_{\e;h}^k\big)+\tfrac2{1+\nu}\big(\lambda_\e^ku_{\e;h}^k+(\Id-\pi_\e^k)[\nabla\cdot h]\big)\qquad\text{in $U$},\]
the standard $\Ld^p$-regularity theory for the Laplace equation yields for all $1<p<\infty$,
\begin{multline*}
\|\nabla u_{\e;h}^k\|_{\Ld^p(U)}\,\le\,K(p)\|(\tfrac2{1+\nu}\Aa(\tfrac\cdot\e)-\Id)\nabla u_{\e;h}^k\|_{\Ld^p(U)}\\
+2K(p)\lambda_\e^k\|u_{\e;h}^k\|_{W^{-1,p}(U)}+2K(p)\|(\Id-\pi_\e^k)[\nabla\cdot h]\|_{W^{-1,p}(U)},
\end{multline*}
where by interpolation the multiplicative constants satisfy
\begin{equation}\label{eq:conv-Kp}
\lim_{p\to2}K(p)= K(2)=1.
\end{equation}
The uniform ellipticity condition~\eqref{eq:ell-as} yields
\[|\tfrac2{1+\nu}\Aa(\tfrac\cdot\e)-\Id|\le\tfrac{1-\nu}{1+\nu},\]
and thus, further appealing to the Poincaré--Sobolev inequality, and inserting the definition of the projection $\pi_\e^k$, the above becomes for all~$2\le p<\infty$,
\begin{equation*}
\|\nabla u_{\e;h}^k\|_{\Ld^p(U)}\,\le\,K(p)\tfrac{1-\nu}{1+\nu}\|\nabla u_{\e;h}^k\|_{\Ld^p(U)}
+C_p\lambda_\e^k\|\nabla u_{\e;h}^k\|_{\Ld^2(U)}+C_p\|h\|_{\Ld^{p}(U)}+C_p\Big|\int_U\nabla g_\e^k\cdot h\Big|.
\end{equation*}
Using the deterministic estimates of Lemma~\ref{lem:lam-g-nabg} on $\lambda_\e^k,g_\e^k$, this yields
\begin{equation}\label{eq:pre-absorb-Meyers}
\|\nabla u_{\e;h}^k\|_{\Ld^p(U)}\,\le\,K(p)\tfrac{1-\nu}{1+\nu}\|\nabla u_{\e;h}^k\|_{\Ld^p(U)}
+C_{k,p}\|\nabla u_{\e;h}^k\|_{\Ld^2(U)}+C_{k,p}\|h\|_{\Ld^{p}(U)}.
\end{equation}
Recalling~\eqref{eq:conv-Kp} and $\tfrac{1-\nu}{1+\nu}<1$, we can choose $C_0\simeq1$ such that
\[K(p)\tfrac{1-\nu}{1+\nu}\,\le\,(\tfrac{1-\nu}{1+\nu})^{1/2}\,<\,1\qquad\text{provided $|p-2|\le\tfrac1{C_0}$}.\]
This allows to absorb the first right-hand side term in~\eqref{eq:pre-absorb-Meyers}: for all $2\le p\le2+\frac1{C_0}$,
\begin{equation*}
\|\nabla u_{\e;h}^k\|_{\Ld^p(U)}\,\lesssim_{k}\,\|\nabla u_{\e;h}^k\|_{\Ld^2(U)}+\|h\|_{\Ld^{p}(U)}.
\end{equation*}
Combined with~\eqref{eq:concl-zL2}, this yields the conclusion.
\end{proof}

\subsection{Convergence of eigenvalues and eigenfunctions}\label{sec:conv}
This section is devoted to the proof of Theorem~\ref{th:conv}.
We start with the following estimates on the fluctuation scaling of eigenvalues, which will be used to upgrade $\Ld^2(\Omega)$-estimates into corresponding $\Ld^q(\Omega)$-estimates.

\begin{lem}[Fluctuation scaling]\label{lem:CLT}
For all $k\ge1$ and $q<\infty$,
\begin{equation*}
\|\lambda_\e^k-\expecmm{\lambda_\e^k}\|_{\Ld^q(\Omega)}\,\lesssim_{k,q}\,\e^\frac d2.
\qedhere
\end{equation*}
\end{lem}

\begin{proof}
In terms of Malliavin calculus, cf.~Proposition~\ref{prop:Mall}, centered moments can be estimated by
\begin{equation}\label{eq:Mall-mom-lambda}
\|\lambda_\e^k-\expecmm{\lambda_\e^k}\|_{\Ld^q(\Omega)}\,\lesssim_q\,\|D\lambda_\e^k\|_{\Ld^q(\Omega;\Hf)}.
\end{equation}
Starting from identity
\[\lambda_\e^k\,=\,\int_U\nabla g_\e^k\cdot\Aa(\tfrac\cdot\e)\nabla g_\e^k,\]
the Malliavin derivative can be written as
\[D\lambda_\e^k\,=\,\int_U\nabla g_\e^k\cdot D\Aa(\tfrac\cdot\e)\nabla g_\e^k+2\int_U\nabla D g_\e^k\cdot \Aa(\tfrac\cdot\e)\nabla g_\e^k.\]
Since the eigenvalue relation and the normalization of $g_\e^k$ ensure that the second right-hand side term is
\[\int_U\nabla D g_\e^k\cdot \Aa(\tfrac\cdot\e)\nabla g_\e^k\,=\,\lambda_\e^k\int_U g_\e^kD g_\e^k\,=\,\tfrac12\lambda_\e^kD\|g_\e^k\|_{\Ld^2(U)}\,=\,0,\]
we deduce
\begin{equation}\label{eq:form-Dlambda0}
D\lambda_\e^k\,=\,\int_U\nabla g_\e^k\cdot D\Aa(\tfrac\cdot\e)\nabla g_\e^k.
\end{equation}
The definition~\eqref{eq:def-A} of $\Aa$ yields
\begin{equation}\label{eq:Da-form}
D_z\Aa(\tfrac x\e)\,=\,\partial a_0(G(z))\,\delta(\tfrac x\e-z)\,=\,\e^d\partial a_0(G(z))\,\delta(x-\e z),
\end{equation}
so that~\eqref{eq:form-Dlambda0} becomes
\begin{equation*}
D\lambda_\e^k\,=\,\e^d\nabla g_\e^k(\e\cdot)\cdot \partial a_0(G(\cdot))\nabla g_\e^k(\e\cdot).
\end{equation*}
Inserting this into~\eqref{eq:Mall-mom-lambda}, and using the integrability condition~\eqref{eq:integr} in form of~\eqref{eq:integr-cons}, we obtain after rescaling,
\begin{equation}\label{eq:mom-lambda}
\|\lambda_\e^k-\expecmm{\lambda_\e^k}\|_{\Ld^q(\Omega)}\,\lesssim_q\,\e^{\frac d2}\,\|[\nabla g_\e^k]_{2;\e}\|_{\Ld^{2q}(\Omega;\Ld^4(U))}^2,
\end{equation}
and the conclusion follows from~\eqref{eq:bnd-nabg}.
\end{proof}

Next, we establish the following preliminary control on eigenvalues, which can be viewed as a version of~\cite[Theorem~2.2]{Kesavan-1} in the random setting (see also~\cite[Theorem~11.4]{JKO94}). The second right-hand side term in the estimate will be absorbed later for small $\e$.

\begin{lem}[Control on eigenvalues]\label{lem:conv-res}
For all $k\ge1$ and $q<\infty$,
\begin{equation*}
\|\lambda_\e^k-\bar\lambda^k\|_{\Ld^q(\Omega)}\,\lesssim_{k,q}\,\e\mu_d(\tfrac1\e)+\Big\|(\lambda_\e^k-\bar\lambda^k)\int_U|g_\e^k-\bar g^k|^2\Big\|_{\Ld^q(\Omega)},
\end{equation*}
where we recall that $\mu_d$ is given by~\eqref{eq:mudr}.
\end{lem}

\begin{proof}
Expanding the gradient, inserting the definition of the flux corrector~$\sigma$, cf.~\eqref{eq:prop-sig}, integrating by parts, using Leibniz' rule, and using the skew-symmetry of $\sigma$, we can write 
\begin{eqnarray}
\lefteqn{\int_U\nabla g_\e^k\cdot\Aa(\tfrac\cdot\e)\nabla\big(\bar g^k+\e\varphi_\alpha(\tfrac\cdot\e)\partial_\alpha\bar g^k\big)}\nonumber\\
&=&\int_U\nabla g_\e^k\cdot\Aa(\tfrac\cdot\e)(\nabla\varphi_\alpha+\ee_\alpha)(\tfrac\cdot\e)\partial_\alpha\bar g^k
+\e\int_U\nabla g_\e^k\cdot(\Aa\varphi_\alpha)(\tfrac\cdot\e)\nabla\partial_\alpha\bar g^k\nonumber\\
&=&\int_U\nabla g_\e^k\cdot\bar\Aa\nabla\bar g^k+\int_U\nabla g_\e^k\cdot(\nabla\cdot\sigma_\alpha)(\tfrac\cdot\e)\partial_\alpha\bar g^k
+\e\int_U\nabla g_\e^k\cdot(\Aa\varphi_\alpha)(\tfrac\cdot\e)\nabla\partial_\alpha\bar g^k\nonumber\\
&=&\int_U\nabla g_\e^k\cdot\bar\Aa\nabla\bar g^k+\e\int_U\nabla g_\e^k\cdot(\Aa\varphi_\alpha-\sigma_\alpha)(\tfrac\cdot\e)\nabla\partial_\alpha\bar g^k,\label{eq:comp-sig}
\end{eqnarray}
and thus, using the eigenvalue relation for $\bar g^k$ to reformulate the first right-hand side term, we get
\begin{equation*}
\int_U\nabla g_\e^k\cdot\Aa(\tfrac\cdot\e)\nabla\big(\bar g^k+\e\varphi_\alpha(\tfrac\cdot\e)\partial_\alpha\bar g^k\big)
\,=\,\bar\lambda_k\int_U g_\e^k\,\bar g^k+\e\int_U\nabla g_\e^k\cdot(\Aa\varphi_\alpha-\sigma_\alpha)(\tfrac\cdot\e)\nabla\partial_\alpha\bar g^k.
\end{equation*}
Since the eigenvalue relation for $g_\e^k$ allows to rewrite the left-hand side as
\begin{equation*}
\int_U\nabla g_\e^k\cdot\Aa(\tfrac\cdot\e)\nabla\big(\bar g^k+\e\varphi_\alpha(\tfrac\cdot\e)\partial_\alpha\bar g^k\big)
\,=\,\lambda_\e^k\int_U g_\e^k \big(\bar g^k+\e\varphi_\alpha(\tfrac\cdot\e)\partial_\alpha\bar g^k\big),
\end{equation*}
we are led to the following identity,
\begin{equation*}
(\lambda_\e^k-\bar\lambda^k)\int_Ug_\e^k\,\bar g^k=\e\int_U\nabla g_\e^k\cdot(\Aa\varphi_\alpha-\sigma_\alpha)(\tfrac\cdot\e)\nabla\partial_\alpha\bar g^k-\e\lambda_\e^k\int_U  \varphi_\alpha(\tfrac\cdot\e) g_\e^k \partial_\alpha\bar g^k.
\end{equation*}
As $g_\e^k$ and $\bar g^k$ are normalized, we note that
\begin{equation}\label{eq:ggnorm}
\int_Ug_\e^k\,\bar g^k\,=\,1-\frac12\int_U|g_\e^k-\bar g^k|^2,
\end{equation}
so that the above can be rewritten as
\begin{multline*}
\lambda_\e^k-\bar\lambda^k\,=\,\e\int_U\nabla g_\e^k\cdot(\Aa\varphi_\alpha-\sigma_\alpha)(\tfrac\cdot\e)\nabla\partial_\alpha\bar g^k-\e\lambda_\e^k\int_U  \varphi_\alpha(\tfrac\cdot\e) g_\e^k \partial_\alpha\bar g^k\\
+\frac12(\lambda_\e^k-\bar\lambda^k)\int_U|g_\e^k-\bar g^k|^2.
\end{multline*}
Taking the $\Ld^q(\Omega)$-norm of both sides, using~\eqref{eq:bnd-lambda}--\eqref{eq:bnd-g}, and using Hölder's inequality, we find for all $p>2$ and $q\ge1$,
\begin{multline*}
\|\lambda_\e^k-\bar\lambda^k\|_{\Ld^q(\Omega)}\,\lesssim_k\,\e\|[\nabla g_\e^k]_{p;\e}\|_{\Ld^{2q}(\Omega;\Ld^p(U))}\|[(\varphi,\sigma)]_{p}\|_{\Ld^{2q}(\Omega;\Ld^p(\frac1\e U))}\|\nabla^2\bar g^k\|_{\Ld^{\frac{p}{p-2}}(U)}\\
+\e\|\varphi\|_{\Ld^q(\Omega;\Ld^2(\frac1\e U))}
+\Big\|(\lambda_\e^k-\bar\lambda^k)\int_U|g_\e^k-\bar g^k|^2\Big\|_{\Ld^q(\Omega)}.
\end{multline*}
Choosing $p>2$ close enough to $2$, the conclusion follows from the corrector estimates of Theorem~\ref{th:cor} and from the estimates of Lemmas~\ref{lem:lam-g-nabg} and~\ref{lem:barg} on $g_\e^k,\bar g^k$.
\end{proof}

The following preliminary estimate provides a control on the convergence of eigenfunctions and on their two-scale expansion error in terms of the convergence of neighboring eigenvalues.

\begin{lem}[Control on eigenfunctions]\label{lem:conv-res2}
For all $k\ge1$ such that $\bar\lambda^k$ is simple, we have for all $q<\infty$ and $\delta>0$,
\begin{multline}\label{eq:contr-eigen}
\|g_\e^k-\bar g^k\|_{\Ld^q(\Omega;\Ld^2(U))}+\|\nabla g_\e^k-(\nabla\varphi_\alpha+\ee_\alpha)(\tfrac\cdot\e)\partial_\alpha\bar g^k\|_{\Ld^q(\Omega;\Ld^2(U))}\\
\,\lesssim_{k,q,\delta}\,(\e\mu_d(\tfrac1\e))^\frac12+\sum_{j=k-1}^{k+1}\|\lambda_\e^j-\bar\lambda^j\|_{\Ld^{q+\delta}(\Omega)}.
\end{multline}
In addition, the following Meyers-type improvement holds: there exists a constant $C_0\simeq1$ such that for all~$2\le p\le2+\frac1{C_0}$ the $\Ld^2(U)$-norms can be replaced by $\Ld^p(U)$-norms in~\eqref{eq:contr-eigen}, at the price of replacing the rate $(\e\mu_d(\tfrac1\e))^{1/2}$ by $(\e\mu_d(\tfrac1\e))^{1/p}$.
 \end{lem}

\begin{proof}
We start with the proof of~\eqref{eq:contr-eigen}.
For a parameter $\rho\in[\e,1]$ to be later optimized (depending on $\e$), set
\[U_\rho:=\{x\in U:\dist(x,\partial U)>\rho\},\qquad\partial_\rho U:=U\setminus U_\rho,\]
and choose a cut-off function $\eta_\rho\in C^\infty_c(U)$ such that
\[\eta_\rho|_{U_\rho}=1,\qquad0\le \eta_\rho\le1,\qquad |\nabla\eta_\rho|\lesssim\tfrac1\rho.\]
Now consider the following truncated two-scale expansion error for the eigenvector $g_\e^k$,
\[w_{\e,\rho}^k\,:=\,g_\e^k-\bar g^k-\e\eta_\rho\varphi_\alpha(\tfrac\cdot\e)\partial_\alpha\bar g^k.\]
The eigenvalue relations for $g_\e^k$ and $\bar g^k$ yield
\[-\nabla\cdot\Aa(\tfrac\cdot\e)\nabla w_{\e,\rho}^k\,=\,\lambda_\e^kg_\e^k-\bar\lambda^k\bar g^k+\nabla\cdot\big((\Aa-\bar\Aa)(\tfrac\cdot\e)\nabla\bar g^k+\Aa(\tfrac\cdot\e)\nabla(\e\eta_\rho\varphi_\alpha(\tfrac\cdot\e)\partial_\alpha\bar g^k)\big),\]
and thus, expanding the gradient in the last right-hand side term and using the skew-symmetric flux corrector~$\sigma$, cf.~\eqref{eq:prop-sig}, we are easily led to
\begin{eqnarray*}
-\nabla\cdot\Aa(\tfrac\cdot\e)\nabla w_{\e,\rho}^k\,=\,
\lambda_\e^kg_\e^k-\bar\lambda^k\bar g^k+\nabla\cdot h_{\e,\rho}^k,
\end{eqnarray*}
in terms of
\begin{equation}\label{eq:def-h-abbr}
h_{\e,\rho}^k\,:=\,(1-\eta_\rho)(\Aa-\bar\Aa)(\tfrac\cdot\e)\nabla\bar g^k+\e(\Aa\varphi_\alpha-\sigma_\alpha)(\tfrac\cdot\e)\nabla(\eta_\rho\partial_\alpha\bar g^k).
\end{equation}
Rewriting this equation as
\begin{equation}\label{eq:wek-eqn}
(-\lambda_\e^k-\nabla\cdot\Aa(\tfrac\cdot\e)\nabla)w_{\e,\rho}^k\,=\,(\lambda_\e^k-\bar\lambda^k)\bar g^k+\e\lambda_\e^k\eta_\rho\varphi_\alpha(\tfrac\cdot\e)\partial_\alpha\bar g^k+\nabla\cdot h_{\e,\rho}^k,
\end{equation}
and appealing to Lemma~\ref{lem:prelim}, together with Poincaré's inequality, we deduce
\[\|\nabla w_{\e,\rho}^k\|_{\Ld^2(U)}\,\lesssim_k\,(\delta_\e^k)^{-1}\Big(|\lambda_\e^k-\bar\lambda^k|+\e\lambda_\e^k\|\varphi_\alpha(\tfrac\cdot\e)\partial_\alpha\bar g^k\|_{\Ld^2(U)}+\|h_{\e,\rho}^k\|_{\Ld^2(U)}\Big).\]
Recalling the definition of $\delta_\e^k=(\lambda_\e^{k+1}-\lambda_\e^k)\wedge(\lambda_\e^k-\lambda_\e^{k-1})\wedge1$, we note that
\[\delta_\e^k \,\ge\, (\bar\lambda^{k+1}-\bar\lambda^k)\wedge(\bar\lambda^{k}-\bar\lambda^{k-1})\wedge1-\sum_{j=k-1}^{k+1}|\lambda_\e^j-\bar\lambda^j|.\]
As $\bar\lambda^k$ is simple, the first right-hand side term is a positive deterministic constant only depending on $d,k,\bar\Aa,U$,
so that the above estimate on $\nabla w_{\e,\rho}^k$ entails
\begin{multline*}
\|\nabla w_{\e,\rho}^k\|_{\Ld^2(U)}\,\lesssim_k\,|\lambda_\e^k-\bar\lambda^k|+\e\lambda_\e^k\|\varphi_\alpha(\tfrac\cdot\e)\partial_\alpha\bar g^k\|_{\Ld^2(U)}+\|h_{\e,\rho}^k\|_{\Ld^2(U)}\\
+\|\nabla w_{\e,\rho}^k\|_{\Ld^2(U)}\sum_{j=k-1}^{k+1}|\lambda_\e^j-\bar\lambda^j|.
\end{multline*}
Taking the $\Ld^q(\Omega)$-norm of both sides, inserting the definition~\eqref{eq:def-h-abbr} of $h_{\e,\rho}^k$, using that $1-\eta_\rho$ and $\nabla\eta_\rho$ are supported in $\partial_\rho U$, which has volume $|\partial_\rho U|\simeq\rho$, and appealing to the corrector estimates of Theorem~\ref{th:cor} and to the estimates of Lemmas~\ref{lem:lam-g-nabg} and~\ref{lem:barg} on $\lambda_\e^k,g_\e^k,\bar g^k$, we easily deduce for all~$q<\infty$ and $\delta>0$,
\begin{equation*}
\|\nabla w_{\e,\rho}^k\|_{\Ld^q(\Omega;\Ld^2(U))}\,\lesssim_{k,q,\delta}\,\rho^\frac12+\e\mu_d(\tfrac1\e)\rho^{-\frac12}
+\sum_{j=k-1}^{k+1}\|\lambda_\e^j-\bar\lambda^j\|_{\Ld^{q+\delta}(\Omega)}.
\end{equation*}
On the one hand, by Poincaré's inequality, using again the corrector estimates of Theorem~\ref{th:cor}, this implies
\begin{eqnarray*}
\|g_\e^k-\bar g^k\|_{\Ld^q(\Omega;\Ld^2(U))}&\lesssim_{k,q}&\e\mu_d(\tfrac1\e)+\|w_{\e,\rho}^k\|_{\Ld^q(\Omega;\Ld^2(U))}\\
&\lesssim_{k,q,\delta}&\rho^\frac12+\e\mu_d(\tfrac1\e)\rho^{-\frac12}
+\sum_{j=k-1}^{k+1}\|\lambda_\e^j-\bar\lambda^j\|_{\Ld^{q+\delta}(\Omega)}.
\end{eqnarray*}
On the other hand, decomposing
\[\nabla w_{\e,\rho}^k\,=\,\nabla g_\e^k-(\nabla\varphi_\alpha+\ee_\alpha)(\tfrac\cdot\e)\partial_\alpha\bar g^k+(1-\eta_\rho)\nabla\varphi_\alpha(\tfrac\cdot\e)\partial_\alpha\bar g^k-\e\varphi_\alpha(\tfrac\cdot\e)\nabla(\eta_\rho\partial_\alpha\bar g^k),\]
we similarly deduce
\begin{equation*}
\|\nabla g_\e^k-(\nabla\varphi_\alpha+\ee_\alpha)(\tfrac\cdot\e)\partial_\alpha\bar g^k\|_{\Ld^q(\Omega;\Ld^2(U))}\,\lesssim_{k,q,\delta}\,\rho^\frac12+\e\mu_d(\tfrac1\e)\rho^{-\frac12}
+\sum_{j=k-1}^{k+1}\|\lambda_\e^j-\bar\lambda^j\|_{\Ld^{q+\delta}(\Omega)}.
\end{equation*}
Now optimizing with respect to $\rho\in[\e,1]$ in these last two estimates, which amounts to choosing $\rho=\e\mu_d(\frac1\e)$, the conclusion~\eqref{eq:contr-eigen} follows.

\medskip\noindent
Finally, the proof of the Meyers-type improvement follows the same line as the above argument, rather starting from the corresponding improvement of Lemma~\ref{lem:prelim} applied to equation~\eqref{eq:wek-eqn}. We skip the details.
\end{proof}

Combining the above different lemmas, we may now conclude the proof of Theorem~\ref{th:conv} by a buckling argument.

\begin{proof}[Proof of Theorem~\ref{th:conv}]
Appealing to Lemma~\ref{lem:CLT} in form of
\begin{eqnarray*}
\|\lambda_\e^k-\bar\lambda^k\|_{\Ld^q(\Omega)}&\le&\|\lambda_\e^k-\expecmm{\lambda_\e^k}\|_{\Ld^q(\Omega)}+|\expecmm{\lambda_\e^k}-\bar\lambda^k|\\
&\lesssim_{k,q}&\e^\frac d2+\|\lambda_\e^k-\bar\lambda^k\|_{\Ld^1(\Omega)},
\end{eqnarray*}
and combining this with Lemma~\ref{lem:conv-res} in $\Ld^1(\Omega)$, we find for all $q<\infty$,
\[\|\lambda_\e^k-\bar\lambda^k\|_{\Ld^q(\Omega)}\,\lesssim_{k,q}\,\e^\frac d2+\e\mu_d(\tfrac1\e)+\Big\|(\lambda_\e^k-\bar\lambda^k)\int_U|g_\e^k-\bar g^k|^2\Big\|_{\Ld^1(\Omega)},\]
and thus, by Hölder's inequality,
\begin{equation}\label{eq:bef-absorb}
\|\lambda_\e^k-\bar\lambda^k\|_{\Ld^q(\Omega)}\,\lesssim_{k,q}\,\e^\frac d2+\e\mu_d(\tfrac1\e)+\|\lambda_\e^k-\bar\lambda^k\|_{\Ld^2(\Omega)}\|g_\e^k-\bar g^k\|_{\Ld^4(\Omega;\Ld^2(U))}^2.
\end{equation}
As it is already known that $g_\e^k\to\bar g^k$ in $\Ld^2(U)$ almost surely as $\e\downarrow0$, the dominated convergence theorem entails $\|g_\e^k-\bar g^k\|_{\Ld^4(\Omega;\Ld^2(U))}\to0$. This qualitative convergence result allows to absorb the last right-hand side term in~\eqref{eq:bef-absorb}, and we deduce the following suboptimal estimate,
\begin{equation}\label{eq:lamb-conv-sub}
\|\lambda_\e^k-\bar\lambda^k\|_{\Ld^q(\Omega)}\,\lesssim_{k,q}\,\e^\frac d2+\e\mu_d(\tfrac1\e).
\end{equation}
Next, inserting this into the result~\eqref{eq:contr-eigen} of Lemma~\ref{lem:conv-res2}, and noting that $\e^{\frac d2}\lesssim(\e\mu_d(\tfrac1\e))^\frac12$, we deduce the optimal convergence rate for eigenfunctions,
\begin{equation*}
\|g_\e^k-\bar g^k\|_{\Ld^q(\Omega;\Ld^2(U))}+\|\nabla g_\e^k-(\nabla\varphi_\alpha+\ee_\alpha)(\tfrac\cdot\e)\partial_\alpha\bar g^k\|_{\Ld^q(\Omega;\Ld^2(U))}\,\lesssim_{k,q,\delta}\,(\e\mu_d(\tfrac1\e))^\frac12.
\end{equation*}
Finally, inserting this bound into the result of Lemma~\ref{lem:conv-res}, we are led to the optimal estimate for eigenvalues,
\[\|\lambda_\e^k-\bar\lambda^k\|_{\Ld^q(\Omega)}\,\lesssim_q\,\e\mu_d(\tfrac1\e)+\|g_\e^k-\bar g^k\|_{\Ld^{2q}(\Omega;\Ld^2(U))}^2\,\lesssim\,\e\mu_d(\tfrac1\e),\]
and the conclusion follows.
\end{proof}

\subsection{Eigenvalue fluctuations}\label{sec:fluct}
This section is devoted to the proof of Theorem~\ref{th:fluct}.
While the fluctuation scaling is already captured in Lemma~\ref{lem:CLT}, we now turn to the characterization of leading-order fluctuations and to their pathwise description~\eqref{eq:pathwise}.

\begin{proof}[Proof of Theorem~\ref{th:fluct}]
The eigenvalue relations for $g_\e^k$ and $\bar g^k$ yield
\[(\lambda_\e^k-\bar\lambda^k)\int_Ug_\e^k\,\bar g^k\,=\,\int_U\nabla\bar g^k\cdot(\Aa(\tfrac\cdot\e)-\bar\Aa)\nabla g_\e^k,\]
which can be rewritten as follows, in view of~\eqref{eq:ggnorm},
\begin{equation}\label{eq:fluct-path}
\lambda_\e^k-\bar\lambda^k\,=\,\int_U\nabla\bar g^k\cdot(\Aa(\tfrac\cdot\e)-\bar\Aa)\nabla g_\e^k\,+\,\frac12(\lambda_\e^k-\bar\lambda^k)\int_U|g_\e^k-\bar g^k|^2,
\end{equation}
where the first right-hand side term involves the so-called homogenization commutator of the eigenfunction $g_\e^k$, in the terminology of~\cite{DGO1}.
Taking inspiration from the fluctuation theory in~\cite{DGO1}, we then replace the homogenization commutator by its two-scale expansion, and we are led to postulating the following approximation,
\[\lambda_\e^k-\expecmm{\lambda_\e^k}\,\sim\,\int_U\Xi^\circ_{\alpha\beta}(\tfrac\cdot\e)\partial_\alpha\bar g^k\partial_\beta\bar g^k,\]
where we recall the definition of the standard homogenization commutator, cf.~\eqref{eq:commut},
\begin{equation*}
\Xi_{\alpha\beta}^\circ\,=\,\ee_\beta\cdot(\Aa-\bar\Aa)(\nabla\varphi_\alpha+\ee_\alpha).
\end{equation*}
It remains to estimate the approximation error. For that purpose, in view of~\eqref{eq:fluct-path}, we can write for all~$q<\infty$,
\begin{equation*}
\Big\|\lambda_\e^k-\expecmm{\lambda_\e^k}-\int_U\Xi^\circ_{\alpha\beta}(\tfrac\cdot\e)\partial_\alpha\bar g^k\partial_\beta\bar g^k\Big\|_{\Ld^q(\Omega)}
\,=\,\|E_\e^k-\expecmm{E_\e^k}\|_{\Ld^q(\Omega)},
\end{equation*}
where we have set for abbreviation
\begin{equation}\label{eq:def-E}
E_\e^k\,:=\,\int_U\nabla\bar g^k\cdot(\Aa(\tfrac\cdot\e)-\bar\Aa)\big(\nabla g_\e^k-(\nabla\varphi_\alpha+\ee_\alpha)(\tfrac\cdot\e)\partial_\alpha\bar g^k\big)
+\frac12(\lambda_\e^k-\bar\lambda^k)\int_U|g_\e^k-\bar g^k|^2.
\end{equation}
Appealing to Malliavin calculus, cf.~Proposition~\ref{prop:Mall}, we get
\begin{equation}\label{eq:err-norm-2sc-exp}
\Big\|\lambda_\e^k-\expecmm{\lambda_\e^k}-\int_U\Xi^\circ_{\alpha\beta}(\tfrac\cdot\e)\partial_\alpha\bar g^k\partial_\beta\bar g^k\Big\|_{\Ld^q(\Omega)}\,\lesssim_q\,\|DE_\e^k\|_{\Ld^q(\Omega;\Hf)}.
\end{equation}
To estimate the right-hand side, we first proceed to a suitable computation of the Malliavin derivative $DE_\e^k$, and we split the proof into two steps.

\medskip
\step1 Proof of
\begin{multline}\label{eq:DE}
DE_\e^k\,=\,\int_U
\big(\nabla g_\e^k+(\nabla\phi_\beta+\ee_\beta)(\tfrac\cdot\e)\partial_\beta\bar g^k\big)
\cdot D\Aa(\tfrac\cdot\e)\big(\nabla g_\e^k-(\nabla\varphi_\alpha+\ee_\alpha)(\tfrac\cdot\e)\partial_\alpha\bar g^k\big)\\
-\e\int_U\nabla(\partial_\alpha\bar g^k\partial_\beta\bar g^k)\cdot\big((\Aa\varphi_\beta+\sigma_\beta)(\tfrac\cdot\e)\nabla D\varphi_\alpha(\tfrac\cdot\e)+ (\varphi_\beta D\Aa)(\tfrac\cdot\e)(\nabla\varphi_\alpha+\ee_\alpha)(\tfrac\cdot\e)\big).
\end{multline}
By definition~\eqref{eq:def-E} of $E_\e^k$, its Malliavin derivative can be decomposed as
\begin{multline}\label{eq:decomp-DE}
DE_\e^k\,=\,\int_U\nabla\bar g^k\cdot D\Aa(\tfrac\cdot\e)\big(\nabla g_\e^k-(\nabla\varphi_\alpha+\ee_\alpha)(\tfrac\cdot\e)\partial_\alpha\bar g^k\big)\\
+\int_U\nabla\bar g^k\cdot(\Aa(\tfrac\cdot\e)-\bar\Aa)\big(\nabla Dg_\e^k-\nabla D\varphi_\alpha(\tfrac\cdot\e)\partial_\alpha\bar g^k\big)
+D\bigg(\frac12(\lambda_\e^k-\bar\lambda^k)\int_U|g_\e^k-\bar g^k|^2\bigg).
\end{multline}
We start by reformulating the last right-hand side term.
As the normalization of the eigenfunction~$g_\e^k$ entails
\[\int_Ug_\e^kDg_\e^k\,=\,\tfrac12D\|g_\e^k\|_{\Ld^2(U)}^2\,=\,0,\]
we can write
\begin{eqnarray}
\lefteqn{D\bigg(\frac12(\lambda_\e^k-\bar\lambda^k)\int_U|g_\e^k-\bar g^k|^2\bigg)}\nonumber\\
&=&(\lambda_\e^k-\bar\lambda^k)\int_U(g_\e^k-\bar g^k)Dg_\e^k+\frac12(D\lambda_\e^k)\int_U|g_\e^k-\bar g^k|^2\nonumber\\
&=&-(\lambda_\e^k-\bar\lambda^k)\int_U\bar g^kDg_\e^k+\frac12(D\lambda_\e^k)\int_U|g_\e^k-\bar g^k|^2.\label{eq:Dlllgg}
\end{eqnarray}
We further reformulate the first right-hand side term in this identity.
Taking the Malliavin derivative of the eigenvalue relation for $g_\e^k$ yields
\[(-\lambda_\e^k-\nabla\cdot\Aa(\tfrac\cdot\e)\nabla) Dg_\e^k\,=\,\nabla\cdot D\Aa(\tfrac\cdot\e)\nabla g_\e^k+(D\lambda_\e^k)g_\e^k,\]
which can be rewritten as follows, after applying $\Id-\pi_\e^k$ to both sides,
\[(-\lambda_\e^k-\nabla\cdot\Aa(\tfrac\cdot\e)\nabla)Dg_\e^k\,=\,(\Id-\pi_\e^k)[\nabla\cdot D\Aa(\tfrac\cdot\e)\nabla g_\e^k].\]
Testing this relation with $\bar g^k$, and using the eigenvalue relation for $\bar g^k$, we get
\begin{equation*}
(\lambda_\e^k-\bar\lambda^k)\int_U\bar g^kDg_\e^k\,=\,\int_U\nabla\bar g^k\cdot(\Aa(\tfrac\cdot\e)-\bar\Aa)\nabla Dg_\e^k+\int_U\nabla(1-\pi_\e^k)[\bar g^k]\cdot D\Aa(\tfrac\cdot\e)\nabla g_\e^k.
\end{equation*}
Combining this with~\eqref{eq:decomp-DE} and~\eqref{eq:Dlllgg}, we find after straightforward simplifications,
\begin{multline}\label{eq:DE-v0}
DE_\e^k\,=\,\int_U\nabla\bar g^k\cdot D\Aa(\tfrac\cdot\e)\big(\nabla g_\e^k-(\nabla\varphi_\alpha+\ee_\alpha)(\tfrac\cdot\e)\partial_\alpha\bar g^k\big)
+\frac12(D\lambda_\e^k)\int_U|g_\e^k-\bar g^k|^2\\
-\int_U(\partial_\alpha\bar g^k)\nabla\bar g^k\cdot(\Aa(\tfrac\cdot\e)-\bar\Aa)\nabla D\varphi_\alpha(\tfrac\cdot\e)
-\int_U\nabla(1-\pi_\e^k)[\bar g^k]\cdot D\Aa(\tfrac\cdot\e)\nabla g_\e^k.
\end{multline}
Next, we further reformulate the last two right-hand side terms in this identity, and we start with the first one. In terms of the skew-symmetric flux corrector~$\sigma$, cf.~\eqref{eq:prop-sig}, using Leibniz' rule, and noting that the Malliavin derivative of the corrector equation yields
\begin{equation}\label{eq:Dphi}
-\nabla\cdot\Aa\nabla D\varphi_\beta=\nabla\cdot D\Aa(\nabla\varphi_\beta+\ee_\beta),
\end{equation}
we easily get
\begin{eqnarray}
\lefteqn{\int_U(\partial_\alpha\bar g^k)\nabla\bar g^k\cdot(\Aa(\tfrac\cdot\e)-\bar\Aa)\nabla D\varphi_\alpha(\tfrac\cdot\e)}\nonumber\\
&=&\int_U(\partial_\alpha\bar g^k\partial_\beta\bar g^k)\,(-\Aa\nabla\varphi_\beta+\nabla\cdot\sigma_\beta)(\tfrac\cdot\e)\cdot\nabla D\varphi_\alpha(\tfrac\cdot\e)\nonumber\\
&=&\int_U(\partial_\alpha\bar g^k\partial_\beta\bar g^k)\,(\nabla\varphi_\beta)(\tfrac\cdot\e)\cdot D\Aa(\tfrac\cdot\e)(\nabla\varphi_\alpha+\ee_\alpha)(\tfrac\cdot\e)\nonumber\\
&&+\,\e\int_U\nabla(\partial_\alpha\bar g^k\partial_\beta\bar g^k)\cdot(\Aa\varphi_\beta+\sigma_\beta)(\tfrac\cdot\e)\nabla D\varphi_\alpha(\tfrac\cdot\e)\nonumber\\
&&+\,\e\int_U\nabla(\partial_\alpha\bar g^k\partial_\beta\bar g^k)\cdot (\varphi_\beta D\Aa)(\tfrac\cdot\e)(\nabla\varphi_\alpha+\ee_\alpha)(\tfrac\cdot\e).
\label{eq:DE-v0/0}
\end{eqnarray}
We turn to the reformulation of the last right-hand side term in~\eqref{eq:DE-v0}. Inserting the definition of the projection $\pi_\e^k$, and using identity~\eqref{eq:ggnorm}, we find
\begin{eqnarray*}
\lefteqn{\int_U\nabla(1-\pi_\e^k)[\bar g^k]\cdot D\Aa(\tfrac\cdot\e)\nabla g_\e^k}\\
&=&\int_U\nabla\bar g^k\cdot D\Aa(\tfrac\cdot\e)\nabla g_\e^k
-\Big(\int_Ug_\e^k\,\bar g^k\Big)\int_U\nabla g_\e^k\cdot D\Aa(\tfrac\cdot\e)\nabla g_\e^k\\
&=&-\int_U(\nabla g_\e^k-\nabla\bar g^k)\cdot D\Aa(\tfrac\cdot\e)\nabla g_\e^k
+\frac12\Big(\int_U|g_\e^k-\bar g^k|^2\Big)\int_U\nabla g_\e^k\cdot D\Aa(\tfrac\cdot\e)\nabla g_\e^k,
\end{eqnarray*}
or alternatively, further recalling the formula~\eqref{eq:form-Dlambda0} for the Malliavin derivative of eigenvalues,
\begin{equation*}
\int_U\nabla(1-\pi_\e^k)[\bar g^k]\cdot D\Aa(\tfrac\cdot\e)\nabla g_\e^k
\,=\,-\int_U(\nabla g_\e^k-\nabla\bar g^k)\cdot D\Aa(\tfrac\cdot\e)\nabla g_\e^k
+\frac12(D\lambda_\e^k)\int_U|g_\e^k-\bar g^k|^2.
\end{equation*}
Inserting this together with~\eqref{eq:DE-v0/0} into~\eqref{eq:DE-v0}, and reorganizing the terms, the claim~\eqref{eq:DE} follows.

\medskip
\step2 Conclusion.\\
In terms of the solution $v_{\e;\alpha}\in H^1_0(U)$ of the auxiliary problem
\begin{equation}\label{eq:veps}
-\nabla\cdot\Aa(\tfrac\cdot\e)\nabla v_{\e;\alpha}\,=\,\nabla\cdot\big((\Aa\varphi_\beta-\sigma_\beta)(\tfrac\cdot\e)\nabla(\partial_\alpha\bar g^k\partial_\beta\bar g^k)\big),\qquad\text{in $U$},
\end{equation}
we can write
\begin{equation*}
\int_U\nabla(\partial_\alpha\bar g^k\partial_\beta\bar g^k)\cdot(\Aa\varphi_\beta+\sigma_\beta)(\tfrac\cdot\e)\nabla D\varphi_\alpha(\tfrac\cdot\e)
\,=\,-\int_U\nabla v_{\e;\alpha}\cdot\Aa(\tfrac\cdot\e)\nabla D\varphi_\alpha(\tfrac\cdot\e),
\end{equation*}
and thus, in view of the Malliavin derivative of the corrector equation, cf.~\eqref{eq:Dphi},
\begin{equation*}
\int_U\nabla(\partial_\alpha\bar g^k\partial_\beta\bar g^k)\cdot(\Aa\varphi_\beta+\sigma_\beta)(\tfrac\cdot\e)\nabla D\varphi_\alpha(\tfrac\cdot\e)
\,=\,\int_U\nabla v_{\e;\alpha}\cdot D\Aa(\tfrac\cdot\e)(\nabla\varphi_\alpha(\tfrac\cdot\e)+\ee_\alpha).
\end{equation*}
Inserting this into~\eqref{eq:DE}, and recalling that the definition~\eqref{eq:def-A} of $\Aa$ yields~\eqref{eq:Da-form}, we are led to
\begin{multline*}
DE_\e^k\,=\,\e^d\big(\nabla g_\e^k(\e\cdot)+(\nabla\phi_\beta+\ee_\beta)\partial_\beta\bar g^k(\e\cdot)\big)
\cdot \partial a_0(G)\big(\nabla g_\e^k(\e\cdot)-(\nabla\varphi_\alpha+\ee_\alpha)\partial_\alpha\bar g^k(\e\cdot)\big)\\
-\e^{d+1}\big(\nabla v_{\e;\alpha}(\e\cdot)+\varphi_\beta\nabla(\partial_\alpha\bar g^k\partial_\beta\bar g^k)(\e\cdot)\big)\cdot \partial a_0(G)(\nabla\varphi_\alpha+\ee_\alpha).
\end{multline*}
Inserting this into~\eqref{eq:err-norm-2sc-exp}, and using the integrability condition~\eqref{eq:integr} in form of~\eqref{eq:integr-cons}, we obtain after rescaling, for all $2< p,q<\infty$,
\begin{eqnarray*}
\lefteqn{\Big\|\lambda_\e^k-\expecmm{\lambda_\e^k}-\int_U\Xi^\circ_{\alpha\beta}(\tfrac\cdot\e)\partial_\alpha\bar g^k\partial_\beta\bar g^k\Big\|_{\Ld^q(\Omega)}}\\
&\lesssim_q&\e^\frac d2\big\|\big[\nabla g_\e^k+(\nabla\varphi_\beta+\ee_\beta)(\tfrac\cdot\e)\partial_\beta\bar g^k\big]_{2;\e}\big\|_{\Ld^{2q}(\Omega;\Ld^\frac{2p}{p-2}(U))}\\
&&\hspace{3cm}\times\big\|\big[\nabla g_\e^k-(\nabla\varphi_\alpha+\ee_\alpha)(\tfrac\cdot\e)\partial_\alpha\bar g^k\big]_{2;\e}\big\|_{\Ld^{2q}(\Omega;\Ld^p(U))}\\
&&+\,\e^{1+\frac d2}\|[\nabla\varphi_\alpha+\ee_\alpha]_2\|_{\Ld^{2q}(\Omega)}\Big(\|[\nabla v_{\e;\alpha}]_{2;\e}\|_{\Ld^2(U;\Ld^{2q}(\Omega))}\\
&&\hspace{6cm}+\big\|\big[\varphi_\beta(\tfrac\cdot\e)\nabla(\partial_\alpha\bar g^k\partial_\beta\bar g^k)\big]_{2;\e}\big\|_{\Ld^2(U;\Ld^{2q}(\Omega))}\Big).
\end{eqnarray*}
Applying the annealed $\Ld^p$-regularity of Theorem~\ref{th:ann-reg} to equation~\eqref{eq:veps}, in form of
\[\|[\nabla v_{\e;\alpha}]_{2;\e}\|_{\Ld^2(U;\Ld^{2q}(\Omega))}\,\lesssim_{p,q}\,\big\|\big[(\Aa\varphi_\beta-\sigma_\beta)(\tfrac\cdot\e)\nabla(\partial_\alpha\bar g^k\partial_\beta\bar g^k)\big]_{2;\e}\big\|_{\Ld^2(U;\Ld^{3q}(\Omega))},\]
and appealing to the corrector estimates of Theorem~\ref{th:cor} and to the estimates of Lemmas~\ref{lem:lam-g-nabg} and~\ref{lem:barg} on $\lambda_\e^k,g_\e^k,\bar g^k$, we get for all $2<p,q<\infty$,
\begin{multline*}
\e^{-\frac d2}\,\Big\|\lambda_\e^k-\expecmm{\lambda_\e^k}-\int_U\Xi^\circ_{\alpha\beta}(\tfrac\cdot\e)\partial_\alpha\bar g^k\partial_\beta\bar g^k\Big\|_{\Ld^q(\Omega)}\,\lesssim_{k,p,q}\,\e\mu_d(\tfrac1\e)\\
+\big\|\big[\nabla g_\e^k-(\nabla\varphi_\alpha+\ee_\alpha)(\tfrac\cdot\e)\partial_\alpha\bar g^k\big]_{2;\e}\big\|_{\Ld^{2q}(\Omega;\Ld^p(U))}.
\end{multline*}
The conclusion~\eqref{eq:pathwise} follows from the Meyers-type result in Lemma~\ref{lem:conv-res2} provided that~$p>2$ is chosen close enough to $2$.
\end{proof}

\section*{Acknowledgements}
We thank Antoine Gloria and Christopher Shirley for motivating discussions on the topic, and we acknowledge financial support from the CNRS-Momentum program.

\bibliographystyle{plain}
\bibliography{biblio}

\end{document}